\documentclass{article}
\usepackage[utf8]{inputenc}
\usepackage{amsmath,amsthm,amssymb,amsfonts}
\usepackage{hyperref,color}
\def\a{\mathbf{a}}
\def\b{\mathbf{b}}
\def\c{\mathbf{c}}
\def\d{\mathbf{d}}

\def\N{\mathbb{N}}
\def\R{\mathbb{R}}

\newtheorem{theorem}{Theorem}[section]
\newtheorem{definition}[theorem]{Definition}
\newtheorem{lemma}[theorem]{Lemma}

\title{Unimodality preservation by ratios of functional series and integral transforms}
\author{Dmitrii Karp\thanks{Department of Mathematics, Holon Institute of Technology, Holon, Israel. Email: dimkrp@gmail.com}, Anna Vishnyakova\thanks{Department of Mathematics, Holon Institute of Technology, Holon, Israel. Email: annalyticity@gmail.com},  Yi Zhang \thanks{Department of Foundational  Mathematics, School of Science, Xi'an Jiaotong-Liverpool University, 
		Suzhou, 215123, China. 
		Email: Yi.Zhang03@xjtlu.edu.cn}}
\date{}

\begin{document}
	
	\maketitle
	
	\begin{abstract}
		An elementary, but very useful lemma due to Biernacki and Krzy\.{z} (1955) asserts that the ratio of two power series inherits monotonicity from that of the sequence of ratios of their respective coefficients.  Over the last two decades it has been realized that, under some additional assumptions, similar claims hold for more general series ratios as well as for unimodality in place of monotonicity. This paper continues this line of research: we consider ratios of general functional series and integral transforms and furnish natural sufficiency conditions for preservation of unimodality by such ratios. Numerous series and integral transforms appearing in applications satisfy our sufficiency conditions, including Dirichlet, factorial and inverse factorial series, Laplace, Mellin and generalized Stieltjes transforms, among many others.  Finally, we illustrate our general results by exhibiting certain statements on monotonicity patterns for ratios of some special functions.   The key role in our considerations is played by the notion of sign regularity. 
	\end{abstract}

	{\vskip 0.25cm}
	
	{\vskip 0.25cm}

	\section{Introduction}
	
	Beginning with the classical lemma due to Biernacki and Krzy\.{z} \cite{BK} there has been a stream of papers studying the monotonicity pattern of the ratio of two power series 
	\begin{equation}\label{eq:powerratio}
		\frac{a(x)}{b(x)}=\frac{\sum\nolimits_{k=0}^{\infty}a_kx^k}{\sum\nolimits_{k=0}^{\infty}b_kx^k}    
	\end{equation}
	as the consequence of the corresponding monotonicity pattern of the sequence $a_k/b_k$.  The original result of  Biernacki and Krzy\.{z} as well as its polynomial version, rediscovered many times, see, for instance, \cite[Lemma~1]{KSJMAA2010}, \cite[Theorem~4.4]{HVV2009}, states that for $b_k>0$, and assuming positive radius of convergence of both series,  the increasing/decreasing character of the sequence $a_k/b_k$ is inherited by the function $a(x)/b(x)$.  This result has been extended in three directions - (1) more complicated monotonicity patterns of $a_k/b_k$; (2) more general functional series, i.e., replacing $x^k$ with some other functional sequence  $\phi_k(x)$; (3) replacing series with an integral (Laplace or more general integral transform).  The earliest result we are aware of in the first direction is given by Belzunce, Ortega and Ruiz in \cite[Lemma~6.4]{BOR2007} in connection with application of probability theory in insurance and finance. These authors stated that unimodality (i.e., one change of monotonicity direction) of the sequence $a_k/b_k$ is also preserved by $a(x)/b(x)$. They did not provide a detailed proof but indicated a method based on the notion of total positivity. Their Lemma~6.4 also assumes infinite radii of convergence of both series in \eqref{eq:powerratio}.  This result was thereafter used in several papers sometimes assuming incorrectly that it is valid in the same form for a finite radius of convergence, see details in \cite[Introduction]{YangChuWangJMAA2015}.  The situation for the power series was clarified in \cite{YangChuWangJMAA2015} by Yang, Chu and Wang, who introduced the function $H_{a,b}(x)=b^2(x)(a(x)/b(x))'/b'(x)$ whose sign at the endpoints of the convergence interval is responsible for distinguishing the cases when $a(x)/b(x)$ is monotonic or unimodal.  More general ratios of functional series of the form
	\begin{equation}\label{eq:generelratio}
		\frac{A(x)}{B(x)}=\frac{\sum\nolimits_{k=0}^{\infty}a_k\varphi_k(x)}{\sum\nolimits_{k=0}^{\infty}b_k\varphi_k(x)}    
	\end{equation}
	have been considered in \cite[Lemma~2.2]{KPMathScan2009}.  Assuming absolute and uniform convergence of both series and their term-wise derivatives on compact subsets of $[0,\infty)$ the authors show that monotonicity of $a_k/b_k$ is inherited by $A(x)/B(x)$ if 
	the sequence of logarithmic derivatives $k\to \varphi_k'(x)/\varphi_k(x)$ is increasing or reversed if $k\to \varphi_k'(x)/\varphi_k(x)$ is decreasing.  If fact, they prove more, namely, that the coefficients of the series expansion of $A'(x)B(x)-A(x)B'(x)$ in the functional sequence $\varphi_k(x)$ are all of the same sign. We also note that the same type of monotonicity rule has been independently established for both power series and more general series by Jameson in \cite{Jameson2021} by using a clever induction argument.

	The result for the ratio \eqref{eq:generelratio}  has been generalized in the recent work \cite{MoaTianAMS2024} to the case when  the sequence $a_k/b_k$ is unimodal. The key conditions for preservation of unimodality given in this work are monotonicity of $x\to{\varphi_k(x)}$ for all $k$ together with monotonicity of $k\to \varphi_k''(x)/\varphi_k'(x)$ for non-negative $x$ in the convergence domain. They also use the function 
	$H_{A,B}(x)=B^2(x)(A(x)/B(x))'/B'(x)$ whose sign at the endpoint of the convergence interval determines whether the ratio $A(x)/B(x)$ is truly unimodal or merely monotonic.  
	
	Finally, in a very recent preprint  \cite{MoaTianArxiv2024} the authors make one further step and consider the ratio \eqref{eq:powerratio}, where the sequence $a_k/b_k$ changes monotonicity twice (for example, it increases, then decreases and then again increases). They present  conditions in terms of the functions $H_{A,B}$ and $H_{A',B'}$  for the ratio $a(x)/b(x)$ to be monotonic or to have one/two changes in monotonicity. 
	
	Note further that similar results for the ratio of the Laplace transforms have been given in \cite{YangTianJMAA2019} and for generic integral transforms in \cite{Qi2022} and \cite{MoaTianAMS2024}.  Independently, in \cite[Lemma~3]{Ferreira-Simon2023} the authors established the monotonicity preservation by the ratios of Laplace and generalized Stieltjes transforms.   In yet another recent work \cite{MaoTianCompRend2023} Mao and Tian give several monotonicity rules for multiple integrals and integrals with variable limits of integration.   In \cite{YangQianChuJIA2017} the authors consider a more general ratio of the form $(P(x)+A(x))/(P(x)+B(x))$, where $A(x)$ and $B(x)$ are power series and in the recent preprint \cite{MaoDuTian2023} $A(x)$ and $B(x)$ are integral transforms on time scales.  Another approach for studying monotonicity patterns of ratios by using derivatives is based on the so-called L'H\^{o}spital's monotonicity rules \cite{Pinelis2001} and can also be used to give yet one more proof of the Biernacki-Krzy\.{z} lemma \cite[Theorem~7.3]{EstradaPavlovic}.  A recent overview of the monotonicity rules of this type including integrals with variable limits of integration and generalizations to time scales can be found in \cite{MaoTianBMMS2024}.

	This seemingly exhaustive treatment of the topic is nevertheless incomplete as some important classes of series (for instance, factorial and inverse factorial series and general Dirichlet series) are excluded by conditions in \cite{MoaTianAMS2024} (at least when applied formally).  Moreover, these authors require both the numerator coefficients $a_k$ and the denominator coefficients $b_k$ in \eqref{eq:generelratio} to be non-negative, while we retain only the requirements on $b_k$ letting the sign of $a_k$ to be arbitrary. We further remark that the underlying structure behind the conditions in \cite{MoaTianAMS2024} is somewhat obscure. The goal of this paper is to present  alternative conditions for unimodality of the ratios of functional series and integral transforms and to reveal that the notion of sign regularity lies at the heart of unimodality preservation property. Furthermore, in some sense our conditions are necessary and sufficient. This is done in the subsequent Section~2.  In Section~3 we give a large list of examples of functional sequences and kernels for which our theory holds. In Section~4 we illustrate some of the results by considering the ratios of generalized hypergeometric and basic hypergeometric functions and ratios of Nuttall's $Q$-functions. Finally, at the end of the paper we included  two appendices containing some of the proofs which would distract the reader should they be included in the main text.  Our approach bears certain resemblance to the developments in \cite{ChoiSmith2017}, where the authors use total positivity to establish conditions for preservation of quasi-concavity by integration with respect to general measures.  In the same paper and references therein one can find applications in probability and economics. Finally, about half-year after the preliminary version of this paper was made available on arXiv, Zakaria Derbazi published a preprint \cite{Derbazi20205} containing generalization of some of the results of this paper to the case of $m$-modular sequences and functions. Note, however, that his definition of unimodality is quite different from ours, so that the results can only be compared after the appropriate adaptation.  
	
\section{Main results}
	
The key role in our considerations will be played by the notion of sign regularity defined as follows \cite[Definition~1.3, Chapter 2, \S1]{KarlinBook}. 

\begin{definition}
Let	$X,Y$ be linearly ordered sets, in this paper - some subsets of $\mathbb{R}$.
A function (kernel) $K: X\times{Y}\to\mathbb{R}$ is called sign regular of order $r\in\mathbb{N}$ ($K\in{SR_r}$) if there is a sequence of signs $\varepsilon_1,\varepsilon_2,\ldots,\varepsilon_r$, each $+1$ or $-1$ such that
\begin{equation}\label{eq:SRkernel}
\varepsilon_m\left|\!\!
\begin{array}{cccc}
	K(x_1,y_1) & K(x_1,y_2) & \ldots & K(x_1,y_m) \\
	K(x_2,y_1) & K(x_2,y_2) & \ldots & K(x_2,y_m) \\
	\vdots&\vdots&\vdots&\ddots \\
	K(x_m,y_1) & K(x_m,y_2) & \ldots & K(x_m,y_m) \\
\end{array}
\!\!\right|\ge0
\end{equation}
for all $m=1,\ldots,r$ and any points $x_1<x_2<\cdots<x_m$,  $y_1<y_2<\cdots<y_m$, $x_i\in{X}$, $y_j\in{Y}$.  If strict inequality holds in the above inequality, the kernel $K$ is called \textit{strictly sign regular} of order $r$, $K\in SSR_r$.  If $\varepsilon_1=\cdots=\varepsilon_r=1$, the kernel is \textit{totally positive}  of order $r$, $K\in{TP_r}$ or \textit{strictly totally positive} ($K\in{STP_r}$) if all determinants above are strictly positive. We will write $r=\infty$ in any of the above notions, if the corresponding property holds for all natural $r$. 
\end{definition}

\textbf{Remark}. Denote by $I$ a convex subset of $\R$ 
containing more than one point (open or closed interval or 
semi-interval, open or closed half-axis or the whole axis). Suppose the set $X$ in the above definition is equal to $\N_0$ and the set $Y=I$, so that the kernel $K(x,y)$ represents a sequence of functions $\varphi_n(y)=K(n,y)$, $n\in\N_0$, $y\in I$. Assume further that each function $\varphi_n(y): I\to\R$ is $(r-1)-$times continuously differentiable, then a sufficient condition for $K(n,y)=\varphi_n(y)\in SSR_r$ is the following: there is sequence of signs $\varepsilon_1,\ldots,\varepsilon_r$ such that for  all $m, 1\leq m \leq r$ and all sequences of indices
$0 \leq n_1 < n_2 < \ldots < n_m$  the Wronskians 
satisfy $\varepsilon_{m} W(\varphi_{n_1}, \varphi_{n_2}, \ldots, \varphi_{n_m})(x)>0$ 
for all $x\in I$.  Sequences of functions satisfying these requirements are known as Descartes' sequences or sequences obeying the Descartes' rule of signs \cite[Chapter V, Problems 87, 90]{PS}. See details below.

Perhaps the most important property of a sign regular kernel  is its variation diminishing property.  Denote by $S^{-}(\lambda_n)_{n=0}^\infty$  the number of sign changes of a real sequence $(\lambda_n)_{n=0}^\infty$ ignoring zeros; for a continuous function $f(x)$ on an interval $I$ define 
$$
S^{-}(f(x))_{x\in I}=\sup S^{-}(f(x_k))_{k=1}^{n},
$$
where supremum is taken over all finite ordered sequences of points in $I$: $x_1<x_2<\cdots<x_n$.

Our main tools are the following theorems, see  \cite[Chapter V, \S3, Theorem~3.1]{KarlinBook} and  \cite[Chapter V, \S1, Theorem~1.5]{KarlinBook} 

\begin{theorem}\label{th:series}
Suppose for a sequence of real continuous functions  $(\varphi_n)_{n=0}^\infty$ defined on a convex subset  $I$ of $\R$ containing more than one point,  the kernel $K(n,x)=\varphi_n(x)\in SR_r$ on $\mathbb{N}_0\times I$ for some $r\ge2$.
Then for every real sequence $(c_n)_{n=0}^\infty$ having no more than $r-1$
sign changes and such that  the series   
$$
f(x)=\sum_{n=0}^\infty c_n \varphi_n(x) 
$$
converges uniformly on all compact subsets of $I$, we have 
$$
S^{-}(f(x))_{x\in I}\le S^{-}(c_n)_{n=0}^\infty.
$$
Moreover, if $S^{-}(f(x))_{x\in I}=S^{-}(c_n)_{n=0}^\infty=k\le{r-1}$, then the 
sign patterns of $f(x)$ and $(c_n)_{n=0}^\infty$  coincide if $\varepsilon_{k}\varepsilon_{k+1}=1$ or are opposite to each other if $\varepsilon_{k}\varepsilon_{k+1}=-1$. Here $\varepsilon_{0}=1$ and $\varepsilon_{m}\det(\phi_{n_i}(x_j))_{i,j=1}^{m}\ge0$ for any finite sequences $n_1<n_2<\cdots<n_m$, $x_1<\cdots<x_m$ with $m\le{r}$.
\end{theorem}

\begin{theorem}\label{th:integral}
	Suppose the kernel $K(x,y):I\times{J}\to\R$ is $SR_r$ for some $r\ge2$, where $I,J$ are convex subsets of $\R$ containing more than one point. Then for any continuous function $g:J\to\R$ having no more than $r-1$
	sign changes and any positive weight $w:J\to(0,\infty)$ such that the integral    
	$$
	f(x)=\int\limits_{J}K(x,y)g(y)w(y)dy 
	$$
	converges uniformly on all compact subsets of $I$, we have 
	$$
	S^{-}(f(x))_{x\in I}\le S^{-}(g(y))_{y\in J}.
	$$
	Moreover, if $	S^{-}(f(x))_{x\in I}=S^{-}(g(y))_{y\in J}=k\le{r-1}$, then the 
	sign patterns of $f(x)$ and $g(y)$  coincide if $\varepsilon_{k}\varepsilon_{k+1}=1$ or are opposite to each other if $\varepsilon_{k}\varepsilon_{k+1}=-1$. Here $\varepsilon_{0}=1$ and $\varepsilon_{m}\det(K(x_i,y_j))_{i,j=1}^{m}\ge0$ for any finite sequences $y_1<y_2<\cdots<y_m$, $x_1<\cdots<x_m$ with $m\le{r}$.
\end{theorem}

\textbf{Remark}. The most general form of the variation diminishing property can be found in \cite[Chapter V, Theorem~3.1]{KarlinBook} with both series and the integral above  replaced by integration with respect to a general positive measure.  The above two theorems are its special cases.   However, to avoid unnecessary complications related to the subtleties of the notions of ``relevant sign changes'' and ``nodal zeros'', we prefer to present the above two cases separately thus restricting our attention to sequences and continuous functions which suffices for our purposes.     

\textbf{Remark}. The ultimate parts of the above theorems concerning the preservation or reversal of the sign patterns is only stated in  \cite[Chapter V, Theorem~3.1]{KarlinBook} for totally positive kernels.  However, for the case of matrices the required formulation is contained in \cite[Chapter V, \S1, Theorem~1.5]{KarlinBook} 
 (beware of a typo in indexing $\varepsilon$) and can be extended to general kernels repeating {\it mutatis mutandis} the proof of  \cite[Chapter V, Theorem~3.1]{KarlinBook}.

Let us fix some terminology.  We will say that a continuous function $g$ defined on a convex real interval $I$ containing more than one point is increasing (decreasing) if $g(x_1)\le g(x_2)$ ($g(x_1)\ge g(x_2)$) for all distinct $x_1<x_2$, $x_1,x_2\in  I$.  This implies that $g$ is either constant or  $g(a+)<g(b-)$ ($g(a+)>g(b-)$), where $a=\inf(I)$, $b=\sup(I)$. Increasing or decreasing function will be called monotonic on $I$.  We will further say that $g$ is unimodal if it  changes monotonicity no more than once. Hence, there are four possibilities: $g$ is increasing, $g$ is decreasing (or both increasing and decreasing if it is constant), $g$ is first increasing, then decreasing, or $g$ is  first decreasing then increasing. If $g$ is unimodal and not monotonic we say that it is strictly unimodal.   Similar definition holds for sequences defined on arbitrary subsets of $N_0$.  To cut it short, a function (or a sequence) will be called unimodal if it has no more than one local extremum at interior points of $I$.  

The application of the above theorems to the ratios of functional series and integral transforms is based on the  following key observation.

\begin{lemma}\label{lm:keylemma} 
	Let $g : I \to \mathbb{R}$ be a given continuous function.
	The function $g$ is unimodal if and only if for every $\lambda \in \mathbb{R}$
	the function $\mu\to f_{\lambda}(\mu)=g(\mu) -\lambda$ has no more than $2$ sign
	changes on $I$, and its sign pattern remains constant for all $\lambda\in\Lambda_2$, where $\Lambda_2$ is the subset of $\mathbb{R}$ such that $\mu\to f_{\lambda}(\mu)$ has exactly 
	$2$ sign changes once $\lambda\in\Lambda_2$.  In a similar fashion,  a real sequence $(d_k)_{k=0}^\infty$ is a unimodal 
	if and only if for every $\lambda \in \mathbb{R}$ the sequence $(d_k - \lambda)_{k=0}^\infty$ 
	has no more than $2$ sign changes 
	and its sign pattern remains constant for all $\lambda\in\Lambda_2$, where $\Lambda_2$ is the subset of $\mathbb{R}$ such that $(d_k - \lambda)_{k=0}^\infty$  has exactly 
	$2$ sign changes once $\lambda\in\Lambda_2$. 
\end{lemma}

\begin{proof}
	We will prove the statement about unimodal functions;   the case of sequences can be handled analogously.
	
	Suppose that $g$ is a unimodal continuous function. Then $g$ has
at most  $2$ intervals of monotonicity. If $g$ is  monotonic  
	on $I,$ then for  all $\lambda \in \mathbb{R}$ the function  
	$f_{\lambda}(\mu)=g(\mu) -\lambda $  has no more than $1$ change of sign
on $I$. Suppose now that $g$ is not monotonic, so that it has exactly $2$ intervals of monotonicity on $I.$ 
	Without loss of generality, we will assume that there exists $\mu_0 \in I$  
	such that $g$ is  increasing and non-constant on the set $I\cap (-\infty, \mu_0],$ 
	and $g$  is  decreasing and non-constant on the set  $ I \cap [\mu_0, + \infty).$ 
	We denote by  $[c,  g(\mu_0)] = \overline{g(I \cap (-\infty, \mu_0])}$, and by 
	$[d, g(\mu_0)] = \overline{g(I\cap [\mu_0, +\infty))}$, where $-\infty \leq c < g(\mu_0), 
	\quad  -\infty \leq d < g(\mu_0) $ (here by $\overline{X}$ we denote the closure of the 
	set $X$). If  $\lambda \leq \min(c, d),$ or $\lambda \geq g(\mu_0),$  then  the function  
	$\mu\to f_{\lambda}(\mu)$ has no sign changes on $I$.  If $\lambda \in (\min(c, d),  \max(c, d)]$,  
	then  the function  $\mu\to f_{\lambda}(\mu)$ has one sign change on $I$. If $\lambda \in(\max(c,d), 
	g(\mu_0))$, then the function  $\mu\to f_{\lambda}(\mu)$ has two sign changes on $I$, so that $\Lambda_2=(\max(c, d)), 
	g(\mu_0))$. For every 
	$\lambda \in \Lambda_2$, the sign pattern of the function  $\mu\to f_{\lambda}(\mu)$ is  $(-, +, -)$ and hence remains constant.

	Suppose now that for every $\lambda \in \mathbb{R}$  the function  $\mu\to f_{\lambda}(\mu)$ 
	has no more than $2$ sign changes on $I$, and its sign pattern remains constant for all $\lambda\in\Lambda_2$.
	By the way of contradiction assume that $g(\mu)$ is not  
	unimodal. Then $g$ has at least $3$ intervals of  monotonicity. 
	Without loss of generality, we can assume that there exist real numbers
	$\alpha, \beta, \gamma, \delta \in I$ with $\alpha < \beta < \gamma < \delta$  such that
\begin{align*}
	&g~\text{is increasing on}~[\alpha, \beta]~\text{and}~g(\alpha) < g(\beta), \\
	&g\text{ is decreasing on }[\beta, \gamma]~\text{and}~g(\beta) > g(\gamma),\\
	&g\text{ is increasing on }[\gamma,  \delta]~\text{and}~g(\gamma) < g(\delta).
\end{align*}
 If $g(\delta) \geq g(\beta)$  then for 
	every $\lambda \in (g(\alpha), g(\beta))$ the function  $f_{\lambda}$ has at least 
	$3$ sign changes on $I$ by Rolle's theorem applied on each subinterval.  This contradicts our hypotheses. Thus, $g(\delta) < g(\beta)$.
	Similar argument yields $g(\alpha) > g(\gamma)$.  For $\lambda \in (g(\alpha), g(\beta))$ the function  
	$\mu\to f_{\lambda}(\mu)$ has $2$ sign changes with the pattern  $(-, +, -)$ on the interval $(\alpha, \gamma)$ implying that it has $2$ sign changes on the whole $I$  by
	our assumption and $(\alpha, \gamma)\subset\Lambda_2$.    For
	$\lambda \in (g(\gamma),  g(\delta))$ the function  $f_{\lambda}$  has $2$ sign
	changes on the interval $(\beta, \delta)$ so that it has $2$ sign changes on the whole $I$ and $(\beta, \delta)\subset \Lambda_2$.  For such 
	values of $\lambda$ the sign pattern of the function $\mu\to f_{\lambda}(\mu)$ is
	$(+, -, +)$. Hence, for different values of $\lambda\in\Lambda_2$ we obtained different sign patterns.  This contradicts the hypothesis of a constant sign pattern for all $\lambda\in\Lambda_2$. So, our assumption that $g$ is not unimodal must be  wrong.
	
\end{proof}

Our main results are the following two theorems.

\begin{theorem}\label{th:seriesratio}
Suppose for a sequence of real continuous functions  $(\varphi_n)_{n=0}^\infty$ defined on a convex interval $I\subseteq\R$,  the kernel $K(n,x)=\varphi_n(x)\in SR_3$ on $\mathbb{N}_0\times I$.  Suppose both series in the following ratio converge uniformly on compact subsets of $I$:
$$
F(x)=\frac{\sum_{k=0}^\infty a_k \varphi_k(x)}{\sum_{k=0}^\infty b_k\varphi_k(x)},
$$
where $a_k\in\R$, $b_k>0$ for all $k\in\N_0$ and the denominator does not vanish for all $x\in I$. If the sequence of quotients $\left\{a_k/b_k\right\}_{k=0}^\infty$ is a unimodal
sequence, then the function $F$ is unimodal.  Moreover, if $F$ is strictly unimodal, then it inherits the monotonicity pattern of  the sequence $\left\{a_k/b_k\right\}_{k=0}^\infty$ if $\varepsilon_2\varepsilon_3>0$ or reverses it if $\varepsilon_2\varepsilon_3<0$, where $\varepsilon_2,\varepsilon_3$ are defined in \eqref{eq:SRkernel}. 
\end{theorem}

A companion theorem for integral transforms takes the form
\begin{theorem}\label{th:integralratio}
	Suppose  $I,J\subseteq\R$ are convex intervals.   
	Let $K(x,t)$ be a $SR_3$ kernel on $I\times{J}\to\R$ and $w:J\to(0,\infty)$ be a positive weight.  Suppose both integrals in the following definition converge uniformly on compact subsets of $I$:
$$
F(x)=\frac{\int_{J}\!K(x,t)A(t)w(t)dt}{\int_{J}\!K(x,t)B(t)w(t)dt},
$$
	where $A:J\to\R$ and $B:J\to(0,\infty)$. If the function $t\to A(t)/B(t)$ is unimodal, then the function $x\to F(x)$ is unimodal.  Moreover, if $F$ is strictly unimodal, then it inherits the monotonicity pattern of $t\to A(t)/B(t)$ if $\varepsilon_2\varepsilon_3>0$ or reverses it if $\varepsilon_2\varepsilon_3<0$, where $\varepsilon_2,\varepsilon_3$ are defined in \eqref{eq:SRkernel}. 
\end{theorem}

\textbf{Remark}.  All scenarios present in the formulation of these theorems can be realized: $F(x)$ maybe increasing, decreasing or strictly unimodal on $I$ in the direction specified by the theorem.  To verify which scenario is realized for a given sequences $a_k$, $b_k$,  it suffices to verify the sign of  the derivative of $F$ at the endpoints of $I$: $F'(a+)$ and/or $F'(b-)$, where $a=\inf{I}$, $b=\sup{I}$.  

We will only prove Theorem~\ref{th:seriesratio}, the proof of Theorem~\ref{th:integralratio} is almost identical.
\begin{proof}
Assume without loss of generality that $\phi_k(x)\ge0$ on $I$ (otherwise factor out $-1$). For each $x\in I$ and arbitrary real $\lambda$ we have
\begin{equation}
	\label{eq6}
	F(x) -\lambda = \frac{\sum_{k=0}^\infty (a_k - \lambda b_k) \varphi_k(x)}{\sum_{k=0}^\infty b_k \varphi_k(x)}.
\end{equation}
The number of sign changes of $F(x) -\lambda$ on $I$ is equal to the number of sign 
changes of the function $G(x, \lambda):= \sum_{k=0}^\infty (a_k - \lambda b_k) 
\varphi_k(x) $ on $I$  (as a function of  $x$) due to the condition $b_k>0$. The function $G$ is a linear 
combination of a sequence of functions $(\varphi_k(x))_{k=0}^\infty$ with coefficients
$$
c_k(\lambda):=(a_k-\lambda b_k)=\Big(\frac{a_k}{b_k}-\lambda\Big)b_k.
$$
Clearly, $\mathrm{sign}(c_k(\lambda))=\mathrm{sign}(a_k/b_k- \lambda)$  due to positivity of $b_k$ .  By our assumptions, the sequence $\left(a_k/b_k\right)_{k=0}^\infty$ is a 
unimodal  sequence, so that by Lemma~\ref{lm:keylemma} $S^{-}((c_k(\lambda))_{k=0}^\infty)\leq 2$.  Moreover, for all 
$\lambda$ such that the sequence of coefficients  $(c_k(\lambda))_{k=0}^\infty$ has exactly $2$ changes of sign its sign pattern  remains constant. Hence, we are in the position to apply Theorem~\ref{th:series} yielding that the  number of sign changes of $F(x) -\lambda$ 
on $I$ is less than or equal to $S^{-}((c_k(\lambda))_{k=0}^\infty)\leq 2$. Moreover, by the same theorem for all $\lambda$ such that the function $F(x) -\lambda$ has exactly $2$ 
sign changes,  the sign pattern of the functions $F(x) -\lambda$ coincides  
with that of the sequences $(c_k(\lambda))_{k=0}^\infty$ if $\varepsilon_3\varepsilon_2=1$ or is reversed  if $\varepsilon_3\varepsilon_2=-1$.  As $\varepsilon_3\varepsilon_2$ does not depend on $\lambda$,  by Lemma~\ref{lm:keylemma}, we conclude that  for every $x\in I $ the function $F(x)$ is a unimodal function of $x$ possessing the monotonicity pattern stated in the theorem.  

\end{proof}

\section{Examples}
\begin{enumerate}
	\item\label{ex:power} $K(x,y)=x^y$ is well-known (and easily seen via the Vandermonde determinant) to be~$STP_{\infty}$ on $(0,\infty)\times(-\infty,\infty)$. This recovers the case of ratios of two power series and ratios of the Mellin transforms 
	$$
	F(y)=\frac{\int_0^{\infty}x^{y-1}A(x)dx}{\int_0^{\infty}x^{y-1}B(x)dx}.
	$$
	If $A(x)/B(x)$ is unimodal, then so is $F(y)$ and the monotonicity pattern of $A(x)/B(x)$  is preserved by $F(y)$ when it is strictly unimodal ($=$ non-monotonic). 
	
	\item\label{ex:exp} The kernel $K(x,y)=\exp(xy)\in STP_{\infty}$ on $(-\infty,\infty)\times(-\infty,\infty)$.  Hence, Theorem~\ref{th:seriesratio} is applicable to ratios of Dirichlet series:
	$$
	F(x)=\frac{\sum_{k=0}^\infty {a_k}e^{\lambda_{k}x}}{\sum_{k=0}^\infty {b_k}e^{\lambda_{k}x}},
	$$
	where $\lambda_0<\lambda_1<\cdots$.
	If $F(x)$ is strictly unimodal it inherits the monotonicity pattern of $\{a_k/b_k\}$. In a similar fashion, the ratios of two-sided and one-sided Laplace transforms (by restricting the domain of integration to $(-\infty,0)$ and changing variable $x\to -x$)
	$$
	F(y)=\frac{\int_{-\infty}^{\infty}e^{xy}A(x)dx}{\int_{-\infty}^{\infty}e^{xy}B(x)dx}~~\text{and}~~G(y)=\frac{\int_{0}^{\infty}e^{-xy}A(x)dx}{\int_{0}^{\infty}e^{-xy}B(x)dx}
	$$
	satisfy Theorem~\ref{th:integralratio}, thus recovering some of the results of \cite{YangTianJMAA2019}.  Note that the kernel $K(x,y)=e^{-xy}$ has the signature $(+,-,-)$ so that the monotonicity pattern of $A(x)/B(x)$ is preserved by $G(y)$ when it is not monotonic.

	\item\label{ex:generalizedStieltjes} The kernel $K(x,y)=(x+y)^{-\alpha}\in STP_{\infty}$ on $(0,\infty)\times(0,\infty)$ for each $\alpha>0$ \cite[(1.6)]{CG83}, so that Theorem~\ref{th:seriesratio} is applicable to the following ratio of series
	$$
	F(x)=\frac{\sum_{k=0}^\infty {a_k}(x+k)^{-\alpha}}{\sum_{k=0}^\infty {b_k}(x+k)^{-\alpha}}
	$$
	and Theorem~\ref{th:integralratio} is applicable to the ratio of the generalized Stieltjes transforms
	$$
	F(y)=\int_{0}^{\infty}\frac{A(x)dx}{(x+y)^{\alpha}}\bigg/\int_{0}^{\infty}\frac{B(x)dx}{(x+y)^{\alpha}}.
	$$

	\item\label{ex:Gamma} The kernel $K(x,y)=\Gamma(x+y)\in STP_{\infty}$ on $(0,\infty)\times(0,\infty)$, see \cite[Theorem~2.11]{GPSR2020} and  Example~\ref{ex:unequalGammaratio} below.  Hence, the kernel $\hat{K}(n,x)=(x)_n=\Gamma(x+n)/\Gamma(x)\in STP_{\infty}$ on  $\mathbb{N}_0\times(0,\infty)$ and Theorem~\ref{th:seriesratio} is applicable to the ratio of factorial series:
	$$
	F(x)=\frac{\sum_{k=0}^\infty {a_k}(x)_k}{\sum_{k=0}^\infty {b_k}(x)_k}.
	$$
	The monotonicity pattern of $\{a_k/b_k\}$ is preserved by $F(x)$ if it is not monotonic.
	Note that 
	$$
	b_0F'(0^+)=\sum_{k=1}^{\infty}b_k(k-1)!\Big(\frac{a_k}{b_k}-\frac{a_0}{b_0}\Big).
	$$
	and $F(x+1)-F(x)$ is eventually negative (positive) if $\{a_k/b_k\}$ is eventually decreasing (increasing) according to Lemma~\ref{lm:factorialasymp}  proved in Appendix~\ref{app:2} to this paper. As $F(x)$ is unimodal if  $\{a_k/b_k\}$ is unimodal by Theorem~\ref{th:seriesratio} it cannot oscillate and  $F(x+1)-F(x)<0$ ($F(x+1)-F(x)>0$) for large $x$ implies that $F(x)$ is decreasing (increasing) for large $x$.   Hence, if $\{a_k/b_k\}$ is first  increasing and then decreasing, then so is $F(x)$  if $F'(0^+)>0$ or $F(x)$ is monotonically decreasing if $F'(0^+)<0$. Similar result holds  if $\{a_k/b_k\}$ is first decreasing  and then increasing. The case $F'(0^+)=0$ may require further investigation.
	
	The beta kernel $K(x,y)=B(x,y)=\Gamma(x+y)/[\Gamma(x)\Gamma(y)]\in STP_{\infty}$ on $(0,\infty)\times(0,\infty)$ by a similar argument, so that one may consider the ratio of the functional series in the sequence $B(x,n)$  on  $(0,\infty)\times\mathbb{N}_0$ instead.
	
	\item\label{ex:incompletegammas}  The argument in \cite[Theorem~2.11]{GPSR2020} for the gamma kernel in the previous example works just as well for the incomplete gamma kernels $K_1(x,y)=\gamma(x+y,\alpha)$ and $K_2(x,y)=\Gamma(x+y,\alpha)$, $\alpha>0$, where 
	$$
	\gamma(z,\alpha)=\int_0^{\alpha}t^{z-1}e^{-t}dt,~~~~\Gamma(z,\alpha)=\int_{\alpha}^{\infty}t^{z-1}e^{-t}dt
	$$
	are the lower and the upper incomplete gamma functions, respectively.  Hence, $K_i(x,y)\in STP_{\infty}$ on $(0,\infty)\times(0,\infty)$ for $i=1,2$.  This implies that the incomplete Pochhammer symbols \cite{SCA2012} are also $STP_{\infty}$ on  $\mathbb{N}_0\times(0,\infty)$:
	$$
    \hat{K}_1(x,n)=(x,\alpha)_n=\frac{\gamma(x+n,\alpha)}{\Gamma(x)}~~\text{and}~~ \hat{K}_2(x,n)=[x,\alpha]_n=\frac{\Gamma(x+n,\alpha)}{\Gamma(x)},
	$$ 
	and Theorem~\ref{th:seriesratio} is applicable to the ratio of '' incomplete factorial series''.  Some interesting special functions of communication theory like Nuttall and Marcum $Q$-functions can be defined using such series \cite{BrychkovITSF2014}.
	
	\item\label{ex:reciprocalGamma} It follows from \cite[Part V, Chap. I, Problem~77]{PS} and \cite[Chapter~3, Lemma~2.2]{KarlinBook} that  the kernel $\hat{K}(n,x)=1/(x)_n\in SSR_{\infty}$ on  $\mathbb{N}_0\times(0,\infty)$. We were unable to locate the more general fact in the literature, which we believe (and verified numerically) to be true: the kernel $K(x,y)=1/\Gamma(x+y)\in SSR_{\infty}$ on $(0,\infty)\times(0,\infty)$.  The property required here is that $\hat{K}(n,x)$ is $SSR_3$ with the signature $(+,-,-)$ can be proved directly by taking the limit $q\to1$ in Lemma~\ref{lm:invqfactorial} proved in Appendix~\ref{app:1}.  Hence, Theorem~\ref{th:seriesratio} is applicable to the ratio of the inverse factorial series:
	$$
	F(x)=\frac{\sum_{k=0}^\infty {a_k}/(x)_k}{\sum_{k=0}^\infty {b_k}/(x)_k}.
	$$
	Monotonicity pattern of $\{a_k/b_k\}$ is preserved by $F(x)$ if it is not monotonic.  Note, further, that 
	$$
	b_0^2F'(x)= \frac{b_0b_1}{x^2}\bigg(\frac{a_0}{b_0}-\frac{a_1}{b_1}\bigg)+O(1/x^3)~~\text{as}~x\to+\infty
	$$
	and, after some calculation (see details in Lemma~\ref{lm:invfactorialzero} in Appendix~\ref{app:2}):
	$$
	\Bigg(\sum_{k=1}^{\infty}\frac{b_k}{(k-1)!}\Bigg)^{2}F'(0+)=\sum_{k=1}^{\infty}\frac{b_0b_k}{(k-1)!}\Big(\frac{a_0}{b_0}-\frac{a_k}{b_k}\Big)+
	\sum_{k=1}^{\infty}\sum_{j=1}^{k-1}\frac{b_kb_j(H_{j-1}-H_{k-1})}{(k-1)!(j-1)!}\Big(\frac{a_k}{b_k}-\frac{a_j}{b_j}\Big),
	$$
	where $H_n$ denotes the $n$-th harmonic number.  Hence, if $a_k/b_k$ is first increasing and eventually decreasing, then $F(x)$ is decreasing for large $x$ and either monotonically decreasing on $(0,\infty)$ if $F'(0+)<0$ or first increasing and then decreasing if  $F'(0+)>0$. The case $F'(0+)=0$ requires further investigation.
	
	\item\label{ex:qGamma} Define the $q$-gamma function for $0<q<1$ by \cite[(21.16)]{KacCheung} (see also \cite[(1.1)]{IsmailMuldoon2013})
	$$
	\Gamma_q(z)=(1-q)^{1-z}\frac{(q;q)_{\infty}}{(q^z;q)_{\infty}},~~~(a;q)_n=\prod\limits_{j=0}^{n-1}(1-aq^{j}).
	$$
	We conjecture that $K(x,y)=\Gamma_{q}(x+y)\in STP_{\infty}$. This probably can be proved using the representation \cite[(21.6)]{KacCheung}, but we will not pursue this in this paper. It is sufficient for our purposes that  $\hat{K}(n,x)=(q^x;q)_n\in STP_3$ on $\mathbb{N}_0\times(0,\infty)$ which is demonstrated in Lemma~\ref{lm:qfactorial} in Appendix~\ref{app:1}. Hence, Theorem~\ref{th:seriesratio} is applicable to the ratio of $q$-factorial series:
	$$
	F(x)=\frac{\sum_{k=0}^\infty {a_k}(q^x;q)_k}{\sum_{k=0}^\infty {b_k}(q^x;q)_k}.
	$$
	If $F(x)$ is strictly unimodal its  monotonicity pattern is inherited from that of the sequence $\{a_k/b_k\}$.
	
	\item\label{ex:reciprocalqGamma} We further conjecture that $K(x,y)=1/\Gamma_{q}(x+y)\in SSR_{\infty}$, $0<q<1$.  The claim sufficient for the purposes of this paper, namely that $\hat{K}(n,x)=1/(q^x;q)_n\in SSR_3$ on $\mathbb{N}_0\times(0,\infty)$ with the sign pattern $(+,-,-)$ is proved in Lemma~\ref{lm:invqfactorial} in Appendix~\ref{app:1}. Hence, Theorem~\ref{th:seriesratio}  is applicable to the ratio of the inverse $q$-factorial series:
	$$
	F(x)=\frac{\sum_{k=0}^\infty {a_k}/(q^x;q)_k}{\sum_{k=0}^\infty {b_k}/(q^x;q)_k}.
	$$
	If $F(x)$ is strictly unimodal its  monotonicity pattern is inherited from that of the sequence $\{a_k/b_k\}$.

	\item\label{ex:Gammaratio} According to \cite[(11)]{KPCMFT2016} we have ($\c=(c_1,\ldots,c_p)$, $\d=(d_1,\ldots,d_p)$)
	$$
	\prod\limits_{i=1}^{p}\frac{\Gamma(x+c_i)}{\Gamma(x+d_i)}=\int\limits_{0}^{1}t^x
	\bigg\{G^{p,0}_{p,p}\left(t\,\,\vline\begin{array}{c}\!\d\\
		\!\c\end{array}\!\!\right)+\delta_1 I_{\{\mu=0\}}\!\bigg\}\frac{dt}{t},
	$$
	where $\mu=\sum_{i=1}^{p}(d_i-c_i)\ge0$, $\delta_1$ denotes the unit mass concentrated at $1$ and $I_{\{\mu=0\}}=1$ if $\mu=0$ and $I_{\{\mu=0\}}=0$ otherwise.  Here $G^{p,0}_{p,p}$ stands for a particular case of Meijer's $G$ function which we prefer to call the Meijer-N{\o}rlund function, see \cite[p.139]{KPCMFT2016} for a definition.  For current purposes we only need one property of this function \cite[section~2]{KPCMFT2016}:
	\begin{equation}\label{eq:vnonnegative}
		\text{if}~v(t)=\sum_{j=1}^{p}(t^{c_j}-t^{d_j})\ge0~\text{for}~t\in(0,1),~\text{then}~G^{p,0}_{p,p}\left(t\,\,\vline\begin{array}{c}\!\d\\\!\c\end{array}\!\!\right)\ge0~\text{for}~t\in(0,1).
	\end{equation}
	Furthermore, convenient sufficient conditions for $v(t)\ge0$ are the following:
	\begin{equation}\label{eq:c-majorized-by-d}
		\begin{split}
			& 0\leq{c_1}\leq{c_2}\leq\cdots\leq{c_p},~~
			0\leq{d_1}\leq{d_2}\leq\cdots\leq{d_p},
			\\
			&~\text{and}~\sum\limits_{i=1}^{k}c_i\leq\sum\limits_{i=1}^{k}d_i~~\text{for}~~k=1,2\ldots,p.
		\end{split}
	\end{equation}
	These and other conditions for non-negativity of $G^{p,0}_{p,p}$ above can be found in \cite[section~2]{KPCMFT2016}.  Assuming that $G$ function is the integrand is non-negative we can write
	$$
	K(x,y)=\prod\limits_{i=1}^{p}\frac{\Gamma(x+y+c_i)}{\Gamma(x+y+d_i)}=\int\limits_{0}^{1}K_1(x,t)K_2(t,y)d\sigma(t),
	$$
	where $K_1(x,t)=t^x$, $K_2(t,y)=t^y$ and $d\sigma(t)$ is a non-negative measure supported on $[0,1]$. Hence, by the basic composition formula \cite[Lemma~1.1, Chapter 3, \S1]{KarlinBook}, the kernel $K(x,y)$ is $STP_\infty$ on $(0,\infty)\times(0,\infty)$, so that the kernel 
	$$
	K(n,x)=\prod\limits_{i=1}^{p}\frac{(x+c_i)_n}{(x+d_i)_n}\in STP_{\infty}~\text{on}~\mathbb{N}_0\times(0,\infty).   
	$$
	Hence, Theorem~\ref{th:seriesratio} is applicable to the ratio
	$$
	F(x)=\frac{\sum\limits_{k=0}^\infty {a_k}\prod\limits_{i=1}^{p}[(c_i+x)_k/(d_i+x)_k]}{\sum\limits_{k=0}^\infty {b_k}\prod\limits_{i=1}^{p}[(c_i+x)_k/(d_i+x)_k]},
	$$
	and if $F(x)$ is not monotonic it inherits monotonicity pattern of $\{a_k/b_k\}$.

	\item\label{ex:unequalGammaratio} We can extend the previous example as follows. Suppose $p\ge0$ and $q\ge 1$. Consider
	$$
	K(x,y)=
	\prod\limits_{i=1}^{p}\frac{\Gamma(x+y+c_i)}{\Gamma(x+y+d_i)}\prod_{j=1}^{q}\Gamma(h_j+x+y).
	$$
	By the previous example if $\sum_{i=1}^{p}(d_i-c_i)>0$, then 
	$$
	\prod\limits_{i=1}^{p}\frac{\Gamma(x+y+c_i)}{\Gamma(x+y+d_i)}=\mathcal{M}(\rho(t))(x+y)
	$$
	where $\mathcal{M}$ is the Mellin transform and 
	$$
	\rho(t)=I_{[0,1]}(t)G^{p,0}_{p,p}\left(t\,\,\vline\begin{array}{c}\!\d\\
		\!\c\end{array}\!\!\right).
	$$
	Hence,
	$$
	K(x,y)=\mathcal{M}(\rho(t)\ast\gamma_1(t)\ast\cdots\ast\gamma_q(t))(x+y),~~~~\gamma_j(t)=t^{h_j}e^{-t},
	$$
	where $\ast$ is the Mellin convolution:
	$$
	h\ast{g}(x)=\int_0^{\infty}h(t)g\Big(\frac{x}{t}\Big)\frac{dt}{t}.
	$$
	Hence, if $v(t)=\sum_{j=1}^{p}(t^{c_j}-t^{d_j})\ge0$ on $(0,1)$ and $h_j\ge0$, $K(x,y)=\mathcal{M}(f)(x+y)$, where the function $f(t)=\rho(t)\ast\gamma_1(t)\ast\cdots\ast\gamma_q(t)\ge0$ on $(0,\infty)$, so that 
	$K(x,y)\in STP_{\infty}$ by the basic composition formula as in the previous example. In particular, the kernel 
	$$
	K(x,y)=\prod_{i=1}^{q}\Gamma(x+y+h_i)\in STP_{\infty}
	$$
	for any $h_i\ge0$ and any natural $q$ and so is $K(n,x)=\prod_{i=1}^{q}(h_i+x)_n$ on $\mathbb{N}_0\times(0,\infty)$.
	
	\item\label{ex:reciprocalGammaratio} Numerical experiments suggest that the kernel  
	$$
	K(x,y)=\frac{\Gamma(c+x+y)}{\Gamma(d+x+y)}
	$$
	is also $SR_{\infty}$ when $c>d$. In particular, it is $SR_3$ with the signature $(+,-,-)$. We were unable to locate any results for sign regularity of this kernel in the literature.  The above kernel is an example of the product of two sign regular Hankel kernels (i.e. the kernels of the form $K(x,y)=F(x+y)$) and we believe that such product is sign regular in general.  The corresponding result for $TP_r$ matrices is known \cite[Theorem~4.5]{FJS2017}, but $F(x_i+y_j)$ is not, in general, a Hankel matrix.  
	Apoorva Khare informed us that he has a method for approximating Hankel kernels by Hankel matrices which may lead to a proof of the required property.   We postpone the details to another publication.

	\item\label{ex:hypergeometric} Suppose $\a=(a_1,\ldots,a_p)$ with $a_i>0$, $i=1,\ldots,p$, and $\mathbf{m}=(m_1,m_2,\ldots,m_p)$ comprises non-negative integers.  Then, according to \cite[Theorem~3.2]{Richards1990}, the kernel 
	$$
	K_1(x,y)={_pF_{p}}\!\left(\begin{array}{c}\a\\\a+\mathbf{m}\end{array}\vline\:xy\right)\in STP_{\infty}~\text{on}~\mathbb{R}^2.
	$$
	Note also that for any positive $\a$, $\b$  the kernel 
	$$
	K_2(x,y)={_pF_{q}}\!\left(\begin{array}{c}\a\\\b\end{array}\vline\:xy\right)\in STP_{\infty}~\text{on}~(0,\infty)\times(0,\infty)
	$$
	due to \cite[p.101, below (1.9) in Chapter 3]{KarlinBook}.  This implies that Theorem~\ref{th:integralratio} is applicable to the ratios of the integral transforms
	$$
	F(x)=\frac{\int_{-\infty}^{\infty}K_1(x,t)A(t)w(t)dt}{\int_{-\infty}^{\infty}K_1(x,t)B(t)w(t)dt}~\text{and}~
	G(x)=\frac{\int_{0}^{\infty}K_2(x,t)A(t)w(t)dt}{\int_{0}^{\infty}K_2(x,t)B(t)w(t)dt}
	$$
	and  the monotonicity pattern of $A(x)/B(x)$ is preserved both by $F(x)$ and $G(x)$ when they are not monotonic.

 \item\label{ex:Bessel}  Consider the kernel $K(n,x)=I_{n}(x)$, where $I_{\nu}$ is the modified Bessel function  \cite[p.217]{MaoTianCompRend2023}, \cite[p.101]{BuchstaberGlutsyuk2019}.  According to \cite[Theorem~1.5]{BuchstaberGlutsyuk2019} this kernel is $STP_{\infty}$ on $\N_0\times(0,\infty)$.  Hence, Theorem~\ref{th:seriesratio} is applicable to the ratio of the modified Bessel function expansions of the form
 $$
 F(x)=\frac{\sum_{k=0}^\infty {a_k}I_k(x)}{\sum_{k=0}^\infty {b_k}I_k(x)}.
 $$
Similarly, restricting the interval $I$ in Theorem~\ref{th:integralratio} to non-negative integers, we conclude that 
 $$
 G(k)=\frac{\int_{0}^{\infty}I_k(t)A(t)w(t)dt}{\int_{0}^{\infty}I_k(x)B(t)w(t)dt}
 $$
 is unimodal on $\N_0$ if so is $A(x)/B(x)$.  If $F(x)$ ($G(k)$) is not monotonic it inherits the monotonicity pattern of $\{a_k/b_k\}$ ($A(x)/B(x)$).
 
 \item\label{ex:Laguerre} Let $L^{\alpha}_j(t)$ be the generalized Laguerre polynomials \cite[p.3]{DOP2019} with $\alpha>-1$ and $j=0,1,\ldots$. According to \cite[Theorem~2]{DOP2019} the matrix $[L^{\alpha}_{j}(t_i)]_{0\le{i,j}\le{n}}$ is $STP$ for each $n\in\N$ and $\alpha>-1$ if the points $t_i$ satisfy
 $t_n<t_{n-1}<\cdots<t_1<t_0<0$.  This implies that ordering the points ascending $\hat{t}_i=t_{n-i}$ the kernel  $K(j,t)=L^{\alpha}_{j}(t)$ is $SSR_{\infty}$ with the sign pattern $(-1)^{r(r-1)/2}$ (where $r$ is the size of the corresponding minor), in particular, it is $SSR_{3}$ with the sign pattern $(+,-,-)$.  Hence, Theorem~\ref{th:seriesratio} is applicable to the ratio of the expansions in the Laguerre polynomials when the argument is outside of the interval of orthogonality:
 $$
  F(x)=\frac{\sum_{k=0}^\infty {a_k}L^{\alpha}_k(x)}{\sum_{k=0}^\infty {b_k}L^{\alpha}_k(x)}
 $$
 where $\alpha>-1$, $x<0$, and if $F(x)$ is not monotonic it inherits the monotonicity pattern of $\{a_k/b_k\}$.
	 
\end{enumerate}

\section{Applications}

Given an $m$-tuple $\c=(c_1,\ldots,c_m)$, we will use the shorthand notation
$(\c+\mu)_k=\prod_{i=1}^{m}(c_i+\mu)_k$ and the product is understood to be $1$ if $m=0$.  Define 
\begin{equation}\label{eq:pFqratio}
F(\mu):=\frac{{_pF_{q}}\!\left(\begin{array}{c}\c+\mu,\a_1\\\d+\mu,\b_1\end{array}\vline\:x\right)}{{_sF_{t}}\!\left(\begin{array}{c}\c+\mu,\b_2\\\d+\mu,\a_2\end{array}\vline\:x\right)}=\frac{\sum\limits_{k=0}^\infty {f_k}(\c+\mu)_k/(\d+\mu)_k}{\sum\limits_{k=0}^\infty {g_k}(\c+\mu)_k/(\d+\mu)_k},
\end{equation}
where $\c$ or $\d$ can be missing. By writing $\a=(\a_1,\a_2)$, $\b=(\b_1,\b_2)$ we will have
$$
\frac{f_k}{g_k}=\frac{(\a_1)_k(\a_2)_k}{(\b_1)_k(\b_2)_k}=\frac{(\a)_k}{(\b)_k}.
$$
 The following fact is well-known and is straightforward to verify: both log-concavity and log-convexity of a sequence imply its unimodality (and if a log-concave sequence is not monotone, then it first increases and then decreases while for log-convex sequence it is reverse).

Log-concavity of $\{f_k/g_k\}$ reduces to the inequality  (similarly for log-convexity)
$$
\frac{(\a+k-1)}{(\b+k-1)}\ge\frac{(\a+k)}{(\b+k)}~\text{for}~k=1,2,\ldots 
$$
A sufficient condition for the above inequality to hold  is the decrease of the rational function 
$$
R_{m,n}(x)=\frac{(\a+x)}{(\b+x)}=\frac{\prod_{k=1}^{m}(a_k+x)}{\prod_{k=1}^{n}(b_k+x)}
$$
 on $(0,\infty)$.  Here $m$, $n$ denote the number of components in $\a$ and $\b$, respectively.   Let $e_j(\a)=e_j(a_1,\dots,a_m)$ be the $j$-th elementary symmetric polynomial. We have a slight extension of Biernaki--Krzy\.z lemma as follows \cite[Lemmas~3,4]{KalmykovKarp2017}:
\begin{lemma}\label{lm:Rmonotonic}
	If $m\le{n}$ and
	\begin{equation*}
		\frac{e_n(\b)}{e_m(\a)}\leq \frac{e_{n-1}(\b)}{e_{m-1}(\a)}\leq \dots \leq \frac{e_{n-m+1}(\b)}{e_1(\a)}\leq e_{n-m}(\b),
	\end{equation*}
	then the function $R_{m,n}(x)$ is monotone decreasing on $(0,\infty)$.   These inequalities  hold, in particular, if  
	$0<a_1\le a_2\le\cdots\le a_m$, $0<b_1\le b_2\le\cdots \le b_{m}$
	and $\sum_{j=1}^{k}a_j\le\sum_{j=1}^{k}b_j$ for $k=1,\ldots,m$.
\end{lemma}
\textbf{Remark.} Clearly, we can apply the above lemma to $1/R_{m,n}(x)$ and get conditions for $R_{m,n}$ to be increasing and hence for $\{f_k/g_k\}$ to be log-convex (and still unimodal).

\begin{theorem}
Suppose the kernel 
$$
K(\mu,n)=\frac{(\c+\mu)_n}{(\d+\mu)_n}\in SR_3~\text{on}~(0,\infty)\times\mathbb{N}_0 
$$
and $\a=(\a_1,\a_2)$, $\b=(\b_1,\b_2)$ satisfy the conditions of Lemma~\ref{lm:Rmonotonic}. Then the function $F(\mu)$ defined in \eqref{eq:pFqratio} is unimodal on $(0,\infty)$.  In particular, this is true if the kernel $K(\mu,n)$ is one of those given in  Examples~\ref{ex:Gamma}, \ref{ex:reciprocalGamma}, \ref{ex:Gammaratio} and \ref{ex:unequalGammaratio}. 
\end{theorem}
In some particular cases of the above theorem we can say more.  Take $\a=(\a_1,\a_2)$, $\b=(\b_1,\b_2)$ satisfying Lemma~\ref{lm:Rmonotonic} and consider 
$$
F_1(\mu)={_pF_{q}}\!\left(\begin{array}{c}\mu,\a_1\\\b_1\end{array}\vline\:x\right)\!\!\bigg/\!{_sF_{t}}\!\left(\begin{array}{c}\mu,\b_2\\\a_2\end{array}\vline\:x\right),
$$
so that we are in the framework of Example~\ref{ex:Gamma}.
We conclude that $F_1(\mu)$ is decreasing for large $\mu$ by Lemma~\ref{lm:factorialasymp} while the sign of $F'(0)$  coincides, by an easy calculation, with that of 
$$
\frac{(\a)_1}{(\b)_1}{_{p+1}F_{q+1}}\!\left(\begin{array}{c}1,\a_1+1\\2,\b_1+1\end{array}\vline\:x\right)-\frac{(\b_2)_1}{(\a_2)_1}{_{s+1}F_{t+1}}\!\left(\begin{array}{c}1,\b_2+1\\2,\a_2+1\end{array}\vline\:x\right).
$$
Hence, if this quantity is negative, then $F_1(\mu)$ is monotone decreasing on $(0,\infty)$. If it is positive, $F_1(\mu)$ is first increasing and then decreasing.  On the other hand, taking 
$$
F_2(\mu)={_pF_{q}}\!\left(\begin{array}{c}\a_1\\\mu,\b_1\end{array}\vline\:x\right)\!\!\bigg/\!{_sF_{t}}\!\left(\begin{array}{c}\b_2\\\mu,\a_2\end{array}\vline\:x\right)
$$
we get a particular case of Example~\ref{ex:reciprocalGamma}.  Thus we conclude that $F_2(\mu)$ is eventually decreasing by the asymptotic formula given in Example~\ref{ex:reciprocalGamma} (still assuming that $\a$, $\b$ satisfy Lemma~\ref{lm:Rmonotonic}).  The derivative $F_2'(0+)$ can be computed by Lemma~\ref{lm:invfactorialzero}.

In a similar fashion, by employing Examples~\ref{ex:qGamma} and \ref{ex:reciprocalqGamma} we get conditions for unimodality of the $q$-hypergeometric ratios
$$
\mu\to{_p\varphi_{q}}\!\left(\begin{array}{c}q^{\mu},q^{\a_1}\\q^{\b_1}\end{array}\vline\:q;x\right)\!\!\bigg/\!{_s\varphi_{t}}\!\left(\begin{array}{c}q^{\mu},q^{\b_2}\\q^{\a_2}\end{array}\vline\:q;x\right)
$$
$$
\mu\to{_p\varphi_{q}}\!\left(\begin{array}{c}q^{\a_1}\\q^{\mu},q^{\b_1}\end{array}\vline\:q;x\right)\!\!\bigg/\!{_s\varphi_{t}}\!\left(\begin{array}{c}q^{\b_2}\\q^{\mu},q^{\a_2}\end{array}\vline\:q;x\right).
$$
Note, however, that for the $q$ case the $m$-tuples $\a$ and $\b$ in Lemma~\ref{lm:Rmonotonic} should replaced by the $m$-tuples:
$$
\hat{\a}=(q^{-a_{1}}-1,q^{-a_{2}}-1,\ldots,q^{-a_{m}}-1),~~~\hat{\:\:\b}=(q^{-b_{1}}-1,q^{-b_{2}}-1,\ldots,q^{-b_{m}}-1)
$$
as explained in \cite[Theorem~1]{KKIssues2018}.

As another example consider the Nuttall $Q$-function \cite{BrychkovITSF2014} which generalizes Marcum $Q$-function and plays a role in communication theory. It is defined by
\begin{equation}\label{eq:NuttallQ}
Q_{\mu,\nu}(a,b)=\int_{b}^{\infty}x^{\mu}e^{-(x^2+a^2)/2}I_{\nu}(ax)dx
\end{equation}
with  $a>0$, $b\ge0$, $\nu>-1$, $\mu>0$ and $I_{\nu}$ standing for the modified Bessel function of the first kind.  Fix $\nu_1>\nu_2$ and $a_1\le a_2$.  Using the standard power series expansion of $I_{\nu}$ \cite[p.217]{MaoTianCompRend2023} we have 
$$
\frac{I_{\nu_1}(a_1x)}{I_{\nu_2}(a_2x)}=2^{\nu_2-\nu_2}a_1^{\nu_1}a_2^{-\nu_2}x^{\nu_1-\nu_2}
\frac{\sum\nolimits_{k=0}^{\infty}a_1^{2k}x^{2k}/[4^kk!\Gamma(k+\nu_1+1)]}{\sum\nolimits_{k=0}^{\infty}a_2^{2k}x^{2k}/[4^kk!\Gamma(k+\nu_2+1)]}.
$$
First we will show that this ratio is a  unimodal function of $x$ on $(0,\infty)$ if $\nu_1-\nu_2=2\ell$, where $\ell\in\N$ (if $\ell=0$ this ratio is immediately seen to be decreasing).  Indeed, in this case the power series in the numerator becomes
$$
\sum\nolimits_{n=\ell}^{\infty}a_1^{2(n-\ell)}x^{2n}/[4^{n-\ell}(n-\ell)!\Gamma(n-\ell+\nu_1+1)]
$$
and the sequence of  ratios of the power series coefficients at $x^{2k}$ for $k=0,1,\ldots$ takes the form
$$
\underbrace{0,\ldots,0}_{\ell~\text{times}},\frac{4^{\ell}\ell!\Gamma(\ell+\nu_2+1)}{a_2^{2\ell}\Gamma(2\ell+\nu_2+1)},\ldots, 4^{\ell}a^{-2\ell}\Big(\frac{a_1}{a_2}\Big)^{2k}\frac{\Gamma(k+1)\Gamma(k+\nu_2+1)}{\Gamma(k-\ell+1)\Gamma(k+\ell+\nu_2+1)},\ldots
$$
It follows from the  $p=2$ case of the integral representation in  Example~\ref{ex:Gammaratio} (see also \cite[Introduction]{KPCMFT2016}) that
$$
k\to\frac{\Gamma(k+1)\Gamma(k+\nu_2+1)}{\Gamma(k-\ell+1)\Gamma(k+\ell+\nu_2+1)}
$$ 
is decreasing for $k\ge\ell$, so that the ratios of the coefficients at $x^{2k}$, $k=0,1,\ldots$ form a sequence which first increases and then decreases. 
By Theorem~\ref{th:seriesratio} or \cite[Corollary~2.3]{YangChuWangJMAA2015} we conclude that the same is true for the ratio $I_{\nu_1}(a_1x)/I_{\nu_2}(a_2x)$. Hence, considering the definition \eqref{eq:NuttallQ} as the (truncated) Mellin transform in the variable $\mu$ we get the following statement:
\begin{theorem}\label{th:NuttallRatio}
	Suppose $b\ge0$, $\nu_1-\nu_2=2\ell\in\N$ and $0<a_1\le a_2$.  Then the ratio
	$$
	\mu\to\frac{Q_{\mu,\nu_1}(a_1,b)}{Q_{\mu,\nu_2}(a_2,b)}
	$$
	is unimodal on $(0,\infty)$.  
\end{theorem}
In view of the reduction formula \cite[p.39]{BrychkovITSF2014} 
$$
Q_{\mu,\nu}(a,0)=2^{(\mu-\nu-1)/2}a^{\nu}e^{-a^2/2}\frac{\Gamma((\mu+\nu+1)/2)}{\Gamma(\nu+1)}{_{1}F_{1}}\!\left(\begin{array}{c}(\mu+\nu+1)/2\\\nu+1\end{array}\vline\:\frac{a^2}{2}\right),
$$
we obtain the corresponding statement for the ratio of the Kummer functions. Furthermore, based on numerical evidence, we believe that the above  ratio of the modified Bessel functions is unimodal not only if $\nu_1-\nu_2=2\ell$ but for all $\nu_1\ge\nu_2>-1$. Notwithstanding the fact that the literature on the ratio of the Bessel functions is vast, we were unable to locate this property and leave it here as the following  conjecture: \\

\noindent\textbf{Conjecture~1.} Suppose $b\ge0$, $\nu_1\ge\nu_2>-1$ and $0<a_1\le a_2$.  Then the ratio
$$
x\to \frac{I_{\nu_1}(a_1x)}{I_{\nu_2}(a_2x)}
$$
is unimodal on $(0,\infty)$. Moreover, if  $\nu_1\ge\nu_2>0$ it is log-concave.

\paragraph{Acknowledgments.} We thank Apoorva Khare for sharing some of his insights regarding sign regularity which was very fruitful and helped to improve the paper.  We are further indebted to anonymous referee for spotting numerous inaccuracies in the first draft of this paper and various insightful suggestions.

\appendix

\setcounter{equation}{0}
\renewcommand{\theequation}{\Alph{section}.\arabic{equation}}

\section{Sign regularity of some $q$-Pochhammer kernels}\label{app:1}

Assume that $k\in\N$, $0<q<1$, and $x, y$ are indeterminates. Let $(x; q)_k$ be the $q$-Pochhammer symbol.  We will need the following easily verifiable identity
\begin{equation}\label{eq:qPochhammerRatio}
	\frac{(x; q)_{m}-(y; q)_{m}}{x-y}=-\sum\limits_{j=0}^{m-1} q^j (x; q)_{j}(y q^{j + 1}; q)_{m-1-j},
\end{equation}	
where $m \in \N_0$, and the straightforward expansion
\begin{equation}\label{eq:DetSum}
	\begin{vmatrix}
		\sum_{i=1}^{m}a_i & \sum_{i=1}^{m}b_i \\[10pt]
		\sum_{j=1}^{n}c_j & \sum_{j=1}^{n}d_j\\
	\end{vmatrix}=\sum_{i=1}^{m}\sum_{j=1}^{n}
	\begin{vmatrix}
		a_i & b_i \\
		c_j & d_j\\
	\end{vmatrix},
\end{equation}	
where $a_i, b_i, c_i, d_i\in\R$ and $m,n\in\N$ .  The proof of the following lemma  is inspired by a calculation in \cite[Theorem~3.2]{ChenLouck}. 

\begin{lemma}\label{lm:qfactorial}
	The kernel $K(\mu,n)=(q^{\mu};q)_n\in STP_3$ on $(0,\infty)\times\mathbb{N}_0$.
\end{lemma}	

\begin{proof}  
	As $K(\mu,n)>0$ we need to verify the signs of the second and the third order determinants.  Verification for the second order determinants is rather straightforward and will be  omitted.  	Suppose $\mu_3 > \mu_2 > \mu_1 >  0$ are reals and $s > j > i \geq 1$ are positive integers. It remains to prove that 
	\begin{equation*}
		\begin{vmatrix}
			(q^{\mu_1}; q)_i & (q^{\mu_2}; q)_i & (q^{\mu_3}; q)_i\\
			(q^{\mu_1}; q)_j & (q^{\mu_2}; q)_j & (q^{\mu_3}; q)_j\\
			(q^{\mu_1}; q)_s & (q^{\mu_2}; q)_s & (q^{\mu_3}; q)_s
		\end{vmatrix}>0.
	\end{equation*}	
	Using rather standard notation for column and row manipulations (say $C_2-C_1\to C_2$ means that the second column should be replaced by the difference of the second and the first column which does not alter the value of the determinant), we compute 
	\begin{multline*}
		\begin{vmatrix}
			(q^{\mu_1}; q)_i & (q^{\mu_2}; q)_i & (q^{\mu_3}; q)_i\\
			(q^{\mu_1}; q)_j & (q^{\mu_2}; q)_j & (q^{\mu_3}; q)_j\\
			(q^{\mu_1}; q)_s & (q^{\mu_2}; q)_s & (q^{\mu_3}; q)_s
		\end{vmatrix}=\Big[\substack{C_2 - C_1 \to C_2 \\[5pt]C_3 - C_1 \to C_3}\Big]
		=\begin{vmatrix}
			(q^{\mu_1}; q)_i & (q^{\mu_2}; q)_i - (q^{\mu_1}; q)_i & (q^{\mu_3}; q)_i - (q^{\mu_1}; q)_i\\
			(q^{\mu_1}; q)_j & (q^{\mu_2}; q)_j - (q^{\mu_1}; q)_j & (q^{\mu_3}; q)_j - (q^{\mu_1}; q)_j\\
			(q^{\mu_1}; q)_s & (q^{\mu_2}; q)_s - (q^{\mu_1}; q)_s & (q^{\mu_3}; q)_s - (q^{\mu_1}; q)_s
		\end{vmatrix}
		\\[10pt]
		=\Big[\substack{R_3 - (q^{\mu_1 + j}; q)_{s - j} R_2 \to R_3}\Big]=\begin{vmatrix}
			(q^{\mu_1}; q)_i & (q^{\mu_2}; q)_i - (q^{\mu_1}; q)_i & (q^{\mu_3}; q)_i - (q^{\mu_1}; q)_i\\
			(q^{\mu_1}; q)_j &  (q^{\mu_2}; q)_j - (q^{\mu_1}; q)_j & (q^{\mu_3}; q)_j -  (q^{\mu_1}; q)_j \\
			0 & (q^{\mu_2}; q)_s - (q^{\mu_1 + j}; q)_{s - j} (q^{\mu_2}; q)_j & (q^{\mu_3}; q)_s - (q^{\mu_1 + j}; q)_{s - j} (q^{\mu_3}; q)_j
		\end{vmatrix}
		\\[10pt]
		=\Big[\substack{R_2 - (q^{\mu_1 + i}; q)_{j - i} R_1 \to R_2}\Big]
		=\begin{vmatrix}
			(q^{\mu_1}; q)_i & (q^{\mu_2}; q)_i - (q^{\mu_1}; q)_i & (q^{\mu_3}; q)_i - (q^{\mu_1}; q)_i\\
			0 &  (q^{\mu_2}; q)_j - (q^{\mu_1 + i}; q)_{j - i} (q^{\mu_2}; q)_i & (q^{\mu_3}; q)_j - (q^{\mu_1 + i}; q)_{j - i} (q^{\mu_3}; q)_i \\
			0 & (q^{\mu_2}; q)_s - (q^{\mu_1 + j}; q)_{s - j} (q^{\mu_2}; q)_j & (q^{\mu_3}; q)_s - (q^{\mu_1 + j}; q)_{s - j} (q^{\mu_3}; q)_j
		\end{vmatrix}
		\\[10pt]
		=(q^{\mu_1}; q)_i \begin{vmatrix}
			(q^{\mu_2}; q)_j - (q^{\mu_1 + i}; q)_{j - i} (q^{\mu_2}; q)_i  &(q^{\mu_3}; q)_j - (q^{\mu_1 + i}; q)_{j - i} (q^{\mu_3}; q)_i\\
			(q^{\mu_2}; q)_s - (q^{\mu_1 + j}; q)_{s - j} (q^{\mu_2}; q)_j & (q^{\mu_3}; q)_s - (q^{\mu_1 + j}; q)_{s - j} (q^{\mu_3}; q)_j
		\end{vmatrix}
\\[10pt]
		=(q^{\mu_1}; q)_i \begin{vmatrix}
			(q^{\mu_2}; q)_i\left[(q^{\mu_2 + i}; q)_{j - i}  - (q^{\mu_1 + i}; q)_{j - i} \right] & (q^{\mu_3}; q)_i\left[(q^{\mu_3 + i}; q)_{j - i}  - (q^{\mu_1 + i}; q)_{j - i} \right] \\
			(q^{\mu_2}; q)_j\left[(q^{\mu_2 + j}; q)_{s - j}  - (q^{\mu_1+ j}; q)_{s - j} \right]  & (q^{\mu_3}; q)_j\left[(q^{\mu_3 + j}; q)_{s - j}  - (q^{\mu_1 + j}; q)_{s - j} \right] 
		\end{vmatrix}
\end{multline*}
\begin{multline*}
		=\text{by \eqref{eq:qPochhammerRatio}} =(q^{\mu_1}; q)_i (q^{\mu_2} - q^{\mu_1}) (q^{\mu_3} - q^{\mu_1})q^i q^j \times
		\\[6pt]
		\begin{vmatrix}
			(q^{\mu_2}; q)_i \sum \limits_{k = 0}^{j - i - 1} q^k (q^{\mu_2 + i}; q)_{k} (q^{\mu_1 + i + k + 1}; q)_{j - i - 1 - k} & (q^{\mu_3}; q)_i \sum \limits_{k = 0}^{j - i - 1} q^k  (q^{\mu_3 + i}; q)_{k} (q^{\mu_1 + i + k + 1}; q)_{j - i - 1 - k}\\[14pt]
			(q^{\mu_2}; q)_j \sum \limits_{\ell = 0}^{s - j - 1} q^\ell (q^{\mu_2 + j}; q)_\ell (q^{\mu_1 + j + \ell + 1}; q)_{s - j - 1 - \ell} & (q^{\mu_3}; q)_j \sum \limits_{\ell = 0}^{s - j - 1} q^\ell (q^{\mu_3 + j}; q)_\ell (q^{\mu_1 + j + \ell + 1}; q)_{s - j - 1 - \ell}
		\end{vmatrix}
		\\[10pt]
		=\text{by \eqref{eq:DetSum}}=
		(q^{\mu_1}; q)_i (q^{\mu_2} - q^{\mu_1}) (q^{\mu_3} - q^{\mu_1}) q^i q^j \times
		\\[6pt]
		\sum \limits_{k = 0}^{j - i - 1} \sum \limits_{\ell = 0}^{s - j - 1}
		\begin{vmatrix}
			q^k (q^{\mu_2}; q)_{i} (q^{\mu_2 + i}; q)_{k} (q^{\mu_1 + i + k + 1}; q)_{j - i - 1 - k} &  q^k (q^{\mu_3}; q)_{i} (q^{\mu_3 + i}; q)_{k} (q^{\mu_1 + i + k + 1}; q)_{j - i - 1 - k}\\[8pt]
			q^\ell (q^{\mu_2}; q)_{ j} (q^{\mu_2 + j}; q)_\ell (q^{\mu_1 + j + \ell + 1}; q)_{s - j - 1 - \ell} & q^\ell (q^{\mu_3}; q)_j (q^{\mu_3 + j}; q)_\ell (q^{\mu_1 + j + \ell + 1}; q)_{s - j - 1 - \ell}
		\end{vmatrix}
		\\[10pt]
		=(q^{\mu_1}; q)_i (q^{\mu_2} - q^{\mu_1}) (q^{\mu_3} - q^{\mu_1}) q^i q^j \times
		\\
		\sum \limits_{k = 0}^{j - i - 1} \sum \limits_{\ell = 0}^{s - j - 1} q^{k + \ell}  (q^{\mu_1 + i + k + 1}; q)_{j - i - 1 - k}  (q^{\mu_1 + j + \ell + 1}; q)_{s - j - 1 - \ell}
		\begin{vmatrix}
			(q^{\mu_2}; q)_{i + k}  &  (q^{\mu_3}; q)_{i + k}\\[8pt]
			(q^{\mu_2}; q)_{j + \ell}  & (q^{\mu_3}; q)_{j + \ell}
		\end{vmatrix}.
	\end{multline*}
	As, clearly, $i+k\le j-1<j+\ell$ each determinant in the summand is positive. 
\end{proof}

\begin{lemma}\label{lm:invqfactorial}
		The kernel $K(\mu,n)=1/(q^{\mu};q)_n\in SSR_3$ on $(0,\infty)\times\mathbb{N}_0$ with the sign pattern $(+,-,-)$.
\end{lemma}	
\begin{proof} 
		As $K(\mu,n)>0$ we need to verify the signs of the second and the third order determinants.  Verification for the second order determinants is rather straightforward and will be  omitted.  Suppose $\mu_3 > \mu_2 > \mu_1 >  0$ are reals and $s > j > i \geq 1$ are positive integers. It remains to prove that 
	\begin{equation} \label{EQ:3rdDet}
		\begin{vmatrix}
			(q^{\mu_1}; q)_i^{-1} & (q^{\mu_2}; q)_i^{-1} & (q^{\mu_3}; q)_i^{-1} \\
			(q^{\mu_1}; q)_j^{-1} & (q^{\mu_2}; q)_j^{-1} & (q^{\mu_3}; q)_j^{-1} \\
			(q^{\mu_1}; q)_s^{-1} & (q^{\mu_2}; q)_s^{-1} & (q^{\mu_3}; q)_s^{-1}
		\end{vmatrix} < 0.
	\end{equation}	
	We compute 
	\begin{multline} \label{EQ:Det1}
		\begin{vmatrix}
			(q^{\mu_1}; q)_i^{-1} & (q^{\mu_2}; q)_i^{-1} & (q^{\mu_3}; q)_i^{-1} \\
			(q^{\mu_1}; q)_j^{-1} & (q^{\mu_2}; q)_j^{-1} & (q^{\mu_3}; q)_j^{-1} \\
			(q^{\mu_1}; q)_s^{-1} & (q^{\mu_2}; q)_s^{-1} & (q^{\mu_3}; q)_s^{-1}
		\end{vmatrix} =\Big[\substack{C_2 - C_1 \to C_2 \\[5pt]C_3 - C_1 \to C_3}\Big]
		=\begin{vmatrix}
			(q^{\mu_1}; q)_i^{-1} & (q^{\mu_2}; q)_i^{-1} - (q^{\mu_1}; q)_i^{-1} & (q^{\mu_3}; q)_i^{-1} - (q^{\mu_1}; q)_i^{-1}\\
			(q^{\mu_1}; q)_j^{-1} & (q^{\mu_2}; q)_j^{-1} - (q^{\mu_1}; q)_j^{-1} & (q^{\mu_3}; q)_j^{-1} - (q^{\mu_1}; q)_j^{-1}\\
			(q^{\mu_1}; q)_s^{-1} & (q^{\mu_2}; q)_s^{-1} - (q^{\mu_1}; q)_s^{-1} & (q^{\mu_3}; q)_s^{-1} - (q^{\mu_1}; q)_s^{-1}
		\end{vmatrix}
		\\[10pt]
		=\Big[\substack{R_3 - (q^{\mu_1 + j}; q)_{s - j}^{-1} R_2 \to R_3}\Big]=\begin{vmatrix}
			(q^{\mu_1}; q)_i^{-1} & (q^{\mu_2}; q)_i^{-1} - (q^{\mu_1}; q)_i^{-1} & (q^{\mu_3}; q)_i^{-1} - (q^{\mu_1}; q)_i^{-1}\\
			(q^{\mu_1}; q)_j^{-1} & (q^{\mu_2}; q)_j^{-1} - (q^{\mu_1}; q)_j^{-1} & (q^{\mu_3}; q)_j^{-1} - (q^{\mu_1}; q)_j^{-1}\\
			0 & (q^{\mu_2}; q)_s^{-1} - (q^{\mu_1 + j}; q)_{s - j}^{-1} (q^{\mu_2}; q)_j^{-1} & (q^{\mu_3}; q)_s^{-1} - (q^{\mu_1 + j}; q)_{s - j}^{-1} (q^{\mu_3}; q)_j^{-1}
		\end{vmatrix}
		\\[10pt]
		=\Big[\substack{R_2 - (q^{\mu_1 + i}; q)_{j - i}^{-1} R_1 \to R_2}\Big]
		=\begin{vmatrix}
			(q^{\mu_1}; q)_i^{-1} & (q^{\mu_2}; q)_i^{-1} - (q^{\mu_1}; q)_i^{-1} & (q^{\mu_3}; q)_i^{-1} - (q^{\mu_1}; q)_i^{-1}\\
			0 &  (q^{\mu_2}; q)_j^{-1} - (q^{\mu_1 + i}; q)_{j - i}^{-1} (q^{\mu_2}; q)_i^{-1} & (q^{\mu_3}; q)_j^{-1} - (q^{\mu_1 + i}; q)_{j - i}^{-1} (q^{\mu_3}; q)_i^{-1} \\
			0 & (q^{\mu_2}; q)_s^{-1} - (q^{\mu_1 + j}; q)_{s - j}^{-1} (q^{\mu_2}; q)_j^{-1} & (q^{\mu_3}; q)_s^{-1} - (q^{\mu_1 + j}; q)_{s - j}^{-1} (q^{\mu_3}; q)_j^{-1}
		\end{vmatrix}
		\\[10pt]
		=(q^{\mu_1}; q)_i^{-1} \begin{vmatrix}
			(q^{\mu_2}; q)_j^{-1} - (q^{\mu_1 + i}; q)_{j - i}^{-1} (q^{\mu_2}; q)_i^{-1} & (q^{\mu_3}; q)_j^{-1} - (q^{\mu_1 + i}; q)_{j - i}^{-1} (q^{\mu_3}; q)_i^{-1} \\
			(q^{\mu_2}; q)_s^{-1} - (q^{\mu_1 + j}; q)_{s - j}^{-1} (q^{\mu_2}; q)_j^{-1} & (q^{\mu_3}; q)_s^{-1} - (q^{\mu_1 + j}; q)_{s - j}^{-1} (q^{\mu_3}; q)_j^{-1}
		\end{vmatrix}
		\\[10pt]
		=(q^{\mu_1}; q)_i^{-1} \begin{vmatrix}
			(q^{\mu_2}; q)_i^{-1} \left[(q^{\mu_2 + i}; q)_{j - i}^{-1}  - (q^{\mu_1 + i}; q)_{j - i}^{-1} \right] &(q^{\mu_3}; q)_i^{-1} \left[(q^{\mu_3 + i}; q)_{j - i}^{-1}  - (q^{\mu_1 + i}; q)_{j - i}^{-1} \right] \\
			(q^{\mu_2}; q)_j^{-1} \left[(q^{\mu_2 + j}; q)_{s - j}^{-1}  - (q^{\mu_1+ j}; q)_{s - j}^{-1} \right]  & (q^{\mu_3}; q)_j^{-1} \left[(q^{\mu_3 + j}; q)_{s - j}^{-1}  - (q^{\mu_1 + j}; q)_{s - j}^{-1} \right] 
		\end{vmatrix}
		\\[10pt]
		=(q^{\mu_1}; q)_i^{-1} \times 
		\begin{vmatrix}
			a & b \\
			c  & d
		\end{vmatrix},
	\end{multline}
	where 
	\begin{align*}
		a & =  (q^{\mu_2}; q)_i^{-1} (q^{\mu_2 + i}; q)_{j - i}^{-1} (q^{\mu_1 + i}; q)_{j- i}^{-1} \left[(q^{\mu_2 + i}; q)_{j - i}  - (q^{\mu_1 + i}; q)_{j - i} \right], \\
		b & = (q^{\mu_3}; q)_i^{-1} (q^{\mu_3 + i}; q)_{j - i}^{-1} (q^{\mu_1 + i}; q)_{j - i}^{-1} \left[(q^{\mu_3 + i}; q)_{j - i}  - (q^{\mu_1 + i}; q)_{j - i} \right], \\
		c & = (q^{\mu_2}; q)_j^{-1} (q^{\mu_2 + j}; q)_{s - j}^{-1} (q^{\mu_1 + j}; q)_{s - j}^{-1} \left[(q^{\mu_2 + j}; q)_{s - j}  - (q^{\mu_1+ j}; q)_{s - j} \right] , \\
		d & = (q^{\mu_3}; q)_j^{-1} (q^{\mu_3 + j}; q)_{s - j}^{-1} (q^{\mu_1 + j}; q)_{s - j}^{-1} \left[(q^{\mu_3 + j}; q)_{s - j}  - (q^{\mu_1 + j}; q)_{s - j} \right] .
	\end{align*}
	By~\eqref{eq:qPochhammerRatio} , we see that~\eqref{EQ:Det1} is equal to
	\begin{equation} \label{EQ:Det2}
		(q^{\mu_1}; q)_i^{-1} (q^{\mu_2} - q^{\mu_1}) (q^{\mu_3} - q^{\mu_1}) q^i q^j \times 
		\begin{vmatrix}
			u & v\\
			f & g
		\end{vmatrix},
	\end{equation}
	where 
	\begin{align*}
		u & = (q^{\mu_2}; q)_i^{-1} (q^{\mu_2 + i}; q)_{j - i}^{-1} (q^{\mu_1 + i}; q)_{j- i}^{-1} \sum \limits_{k = 0}^{j - i - 1} q^k (q^{\mu_2 + i}; q)_{k} (q^{\mu_1 + i + k + 1}; q)_{j - i - 1 - k}, \\
		v & = (q^{\mu_3}; q)_i^{-1} (q^{\mu_3 + i}; q)_{j - i}^{-1} (q^{\mu_1 + i}; q)_{j - i}^{-1} \sum \limits_{k = 0}^{j - i - 1} q^k  (q^{\mu_3 + i}; q)_{k} (q^{\mu_1 + i + k + 1}; q)_{j - i - 1 - k}, \\
		f & = (q^{\mu_2}; q)_j^{-1} (q^{\mu_2 + j}; q)_{s - j}^{-1} (q^{\mu_1 + j}; q)_{s - j}^{-1} \sum \limits_{\ell = 0}^{s - j - 1} q^\ell (q^{\mu_2 + j}; q)_\ell (q^{\mu_1 + j + \ell + 1}; q)_{s - j - 1 - \ell}, \\
		g & = (q^{\mu_3}; q)_j^{-1} (q^{\mu_3 + j}; q)_{s - j}^{-1} (q^{\mu_1 + j}; q)_{s - j}^{-1} \sum \limits_{\ell = 0}^{s - j - 1} q^\ell (q^{\mu_3 + j}; q)_\ell (q^{\mu_1 + j + \ell + 1}; q)_{s - j - 1 - \ell}.
	\end{align*}
	In view of $(q^{\mu_2};q)_i^{-1}(q^{\mu_2+i};q)_{j-i}^{-1}=(q^{\mu_2};q)_j^{-1}$ and $ (q^{\mu_3}; q)_i^{-1} (q^{\mu_3 + i}; q)_{j - i}^{-1}=(q^{\mu_3}; q)_j^{-1}$ and employing \eqref{eq:DetSum}, we conclude that that~\eqref{EQ:Det2} is equal to
	\begin{multline*}
		(q^{\mu_1}; q)_i^{-1} (q^{\mu_2}- q^{\mu_1})(q^{\mu_3}-q^{\mu_1})q^{i+j}
		(q^{\mu_1+i};q)_{j-i}^{-1}(q^{\mu_1+j};q)_{s-j}^{-1} (q^{\mu_2}; q)_j^{-1}(q^{\mu_3};q)_j^{-1}
		\\[6pt]
		\times\sum\limits_{k=0}^{j-i-1}\sum\limits_{\ell=0}^{s-j-1}q^{k+\ell}(q^{\mu_1+i+k+1}; q)_{j-i-1-k}  (q^{\mu_1+j+\ell+1};q)_{s-j-1-\ell} \Delta_{k\ell}, 
	\end{multline*}
	where
	\begin{align*}
	\Delta_{k\ell} & = 
		\begin{vmatrix}
			(q^{\mu_2 + i}; q)_{k}  &  (q^{\mu_3 + i}; q)_{k}\\[8pt]
			(q^{\mu_2 + j}; q)_{s - j}^{-1} (q^{\mu_2 + j}; q)_{\ell}  & (q^{\mu_3 + j}; q)_{s - j}^{-1} (q^{\mu_3 + j}; q)_{\ell}
		\end{vmatrix}\\[6pt]
		& = (q^{\mu_3 + j}; q)_{s - j}^{-1} (q^{\mu_2 + i}; q)_k (q^{\mu_3 + j}; q)_\ell - (q^{\mu_2 + j}; q)_{s - j}^{-1} (q^{\mu_3 + i}; q)_{k} (q^{\mu_2 + j}; q)_\ell \\[6pt]
		& = (q^{\mu_3 + j}; q)_{s - j}^{-1}  (q^{\mu_2 + j}; q)_{s - j}^{-1} \left[ (q^{\mu_2 + j}; q)_{s - j} (q^{\mu_2 + i}; q)_k (q^{\mu_3 + j}; q)_\ell - (q^{\mu_3 + j}; q)_{s - j} (q^{\mu_3 + i}; q)_{k} (q^{\mu_2 + j}; q)_\ell  \right] \\[6pt]
		& =  (q^{\mu_3 + j}; q)_{s - j}^{-1}  (q^{\mu_2 + j}; q)_{s - j}^{-1} (q^{\mu_3 + j}; q)_\ell (q^{\mu_2 + j}; q)_\ell 
		\\[6pt]
		& \times\left[ (q^{\mu_2 + j + \ell}; q)_{s - j - \ell} (q^{\mu_2 + i}; q)_k  - (q^{\mu_3 + j + \ell}; q)_{s - j - \ell} (q^{\mu_3 + i}; q)_{k}\right] < 0
	\end{align*}
	in view of $\mu_2<\mu_3$.  This proves \eqref{EQ:3rdDet}. 
\end{proof}

\setcounter{equation}{0}
\renewcommand{\theequation}{\Alph{section}.\arabic{equation}}

\section{Derivative asymptotics for factorial and inverse factorial series}\label{app:2}
\begin{lemma}\label{lm:factorialasymp}
	Suppose $b_k>0$ for all $k\ge0$ and 
	$$
	\frac{a_0}{b_0}\le \frac{a_1}{b_1}\le\cdots\le \frac{a_m}{b_m},~~~\frac{a_m}{b_m}\ge\frac{a_{m+1}}{b_{m+1}}\ge\frac{a_{m+2}}{b_{m+2}}\cdots 
	$$
	and at least one inequality in each chain is strict.   Define
	$$
	F(x)=\frac{\sum_{k=0}^{\infty}a_k(x)_k}{\sum_{k=0}^{\infty}b_k(x)_k}=\frac{A(x)}{B(x)},
	$$
	and assume that each series converges uniformly on all compact subsets of $\mathbb{R}$.  Then $F(x+1)-F(x)<0$ for all sufficiently large $x$.		
\end{lemma}

\begin{proof}
	We have 
	$$
	F(x+1)-F(x)=\frac{A(x+1)B(x)-A(x)B(x+1)}{B(x)B(x+1)}.
	$$	
	In view of $(x+1)_k=(x)_{k}(x+k)/x$, we obtain
	\begin{multline*}
		A(x+1)B(x)-A(x)B(x+1)=\frac{1}{x}\sum_{k=0}^{\infty}a_k(x)_{k}(x+k)\sum_{n=0}^{\infty}b_n(x)_n-\frac{1}{x}\sum_{k=0}^{\infty}a_k(x)_k\sum_{n=0}^{\infty}b_n(x)_{n}(x+n)
		\\
		=\frac{1}{x}\sum_{k,n=0}^{\infty}a_kb_n(x)_{k}(x)_{n}(k-n)=-\frac{1}{x}\sum_{k,n=0}^{\infty}a_nb_k(x)_{k}(x)_{n}(k-n)
		\\
		=\frac{1}{2x}\sum_{k,n=0}^{\infty}b_kb_n(x)_{k}(x)_{n}(k-n)\bigg(\frac{a_k}{b_k}-\frac{a_n}{b_n}\bigg)
		=\frac{S(x)}{2x},
	\end{multline*}	
	so that 
	$$
	F(x+1)-F(x)=\frac{S(x)}{2xB(x)B(x+1)}
	$$	
	and the sign of the left hand side coincides with that of $S(x)$. Decomposing yields:
	\begin{multline*}
		S(x)=\sum_{k,n=0}^{m-1}b_kb_n(x)_{k}(x)_{n}(k-n)\bigg(\frac{a_k}{b_k}-\frac{a_n}{b_n}\bigg)+
		\sum_{n=0}^{m-1}\sum_{k=m}^{\infty}b_kb_n(x)_{k}(x)_{n}(k-n)\bigg(\frac{a_k}{b_k}-\frac{a_n}{b_n}\bigg)
		\\
		+\sum_{k=0}^{m-1}\sum_{n=m}^{\infty}b_kb_n(x)_{k}(x)_{n}(k-n)\bigg(\frac{a_k}{b_k}-\frac{a_n}{b_n}\bigg)
		+\sum_{k=m}^{\infty}\sum_{n=m}^{\infty}b_kb_n(x)_{k}(x)_{n}(k-n)\bigg(\frac{a_k}{b_k}-\frac{a_n}{b_n}\bigg)
		\\
		=\sum_{k,n=0}^{m-1}b_kb_n(x)_{k}(x)_{n}(k-n)\bigg(\frac{a_k}{b_k}-\frac{a_n}{b_n}\bigg)+
		2\sum_{k=m}^{\infty}b_k(x)_{k}\sum_{n=0}^{m-1}b_n(x)_{n}(k-n)\bigg(\frac{a_k}{b_k}-\frac{a_n}{b_n}\bigg)
		\\
		+\sum_{k=m}^{\infty}b_k(x)_{k}\sum_{n=m}^{\infty}b_n(x)_{n}(k-n)\bigg(\frac{a_k}{b_k}-\frac{a_n}{b_n}\bigg)
		=\sum_{k,n=0}^{m-1}b_kb_n(x)_{k}(x)_{n}(k-n)\bigg(\frac{a_k}{b_k}-\frac{a_n}{b_n}\bigg)
		\\
		+
		\sum_{k=m}^{\infty}b_k(x)_{k}
		\Bigg[
		\sum_{n=0}^{m-1}2b_n(x)_{n}(k-n)\bigg(\frac{a_k}{b_k}-\frac{a_n}{b_n}\bigg)
		-\sum_{n=m}^{\infty}b_n(x)_{n}(n-k)\bigg(\frac{a_k}{b_k}-\frac{a_n}{b_n}\bigg)\Bigg].
	\end{multline*}	
	From the last expression we get
	\begin{multline*}
		\frac{S(x)}{[(x)_{m}]^2}=\sum_{k,n=0}^{m-1}b_kb_n\frac{(x)_{k}(x)_{n}}{(x)_{m}(x)_{m}}(k-n)\bigg(\frac{a_k}{b_k}-\frac{a_n}{b_n}\bigg)
		\\
		+
		\sum_{k=m}^{\infty}b_k\frac{(x)_{k}}{(x)_{m}}
		\Bigg[
		\sum_{n=0}^{m-1}2b_n\frac{(x)_{n}}{(x)_m}(k-n)\bigg(\frac{a_k}{b_k}-\frac{a_n}{b_n}\bigg)
		-\sum_{n=m}^{\infty}b_n\frac{(x)_{n}}{(x)_m}(n-k)\bigg(\frac{a_k}{b_k}-\frac{a_n}{b_n}\bigg)\Bigg].
	\end{multline*}	
	Clearly, we can make the first sum as small as we wish  by making $x$ sufficiently large.  Next, the  expression in brackets in the second sum is eventually negative (for large $x$) , since the first sum tends to zero while the second sum is positive for all $x>0$ and is monotonically increasing, because for some (at least one) $n>m$ the coefficient at $(x)_n/(x)_m$ is strictly positive, so that the second sum in brackets tends to $+\infty$ as $x$ increases.  This implies that $S(x)$ is negative for sufficiently large $x$ which establishes our claim.
\end{proof}

\begin{lemma}\label{lm:invfactorialzero}
	Define
	$$
	F(x)=\frac{\sum_{k=0}^\infty {a_k}/(x)_k}{\sum_{k=0}^\infty {b_k}/(x)_k}.
	$$
	and assume that each series converges uniformly on all compact subsets of $(0,\infty)$.  Then 
	\begin{equation}\label{eq:invfactorialzero}
	\Bigg(\sum_{k=1}^{\infty}\frac{b_k}{(k-1)!}\Bigg)^{2}F'(0+)=\sum_{k=1}^{\infty}\frac{b_0b_k}{(k-1)!}\Big(\frac{a_0}{b_0}-\frac{a_k}{b_k}\Big)+
	\sum_{k=1}^{\infty}\sum_{j=1}^{k-1}\frac{b_kb_j(H_{j-1}-H_{k-1})}{(k-1)!(j-1)!}\Big(\frac{a_k}{b_k}-\frac{a_j}{b_j}\Big),
	\end{equation}
	where $H_n$ denotes the $n$-th harmonic number.
\end{lemma}
\begin{proof}
	In in view of the elementary differentiation formula 
	$$
	\frac{d}{dx}\frac{1}{(x)_k}=\frac{\psi(x)-\psi(x+k)}{(x)_k},
	$$
	where $\psi(z)=\Gamma'(z)/\Gamma(z)$ stands for digamma function, we have:
	$$
	x^2\Bigg(\sum_{k=1}^{\infty}\frac{b_k}{(x)_k}\Bigg)^{2}F'(x)
	=\sum\limits_{k,n=0}^{\infty}b_ka_n\frac{(\psi(x+k)-\psi(x+n))x^2}{(x)_k(x)_n}.
	$$
	As, clearly, $(x)_k/x=(x+1)_{k-1}\to(k-1)!$ and $x(\psi(x+n)-\psi(x))\to1$ as $x\to$ for $n=1,2,\ldots$, by letting $x\to0$ we obtain
	$$
	\Bigg(\sum_{k=1}^{\infty}\frac{b_k}{(k-1)!}\Bigg)^{2}F'(0+)=a_0\sum_{k=1}^{\infty}\frac{b_k}{(k-1)!}
	-b_0\sum_{k=1}^{\infty}\frac{a_k}{(k-1)!}+\sum\limits_{k,n=1}^{\infty}b_ka_n\frac{(\psi(k)-\psi(n))}{(k-1)!(n-1)!}.
	$$
	Finally using $\psi(n)=H_{n-1}-\gamma$, where $H_n$ denotes the $n$-th harmonic number and $\gamma$ is Euler-Mascheroni constant, we arrive at \eqref{eq:invfactorialzero}.
\end{proof}

\end{document}